\documentclass[submission]{eptcs}
\usepackage{amsthm}
\usepackage{amsfonts,amsmath,amssymb}
\usepackage{graphicx} 
\usepackage{tikz}

\usepackage{dsfont}
\usepackage{cancel}
\usepackage[colorinlistoftodos]{todonotes}
\usepackage{dirtytalk}
\usepackage{tabto}
\usepackage{xcolor}
\usepackage{comment}

\usepackage{mathtools}
\usepackage{mathrsfs}
\usepackage{hyperref}
\usepackage{preamb}

\usepackage{iftex}

\ifpdf
  \usepackage{underscore}       
  \usepackage[T1]{fontenc}      
\else
  \usepackage{breakurl}   
\fi

\title{On Local Finiteness of Modal K4 Algebras}
\author{Gabriel Agnew
\institute{
New Mexico State University
\\
Las Cruces, United States}
\email{gvagnew@nmsu.edu}
}

\newcommand{\titlerunning}{On Local Finiteness of Modal K4 Algebras}
\newcommand{\authorrunning}{G. Agnew}

\hypersetup{
  bookmarksnumbered,
  pdftitle    = {\titlerunning},
  pdfauthor   = {\authorrunning},
}
\begin{document}
\maketitle

\begin{abstract}
We study local finiteness for modal $K4$ algebras via the tunability of their dual general frames.
In particular, we provide a sufficient condition for modal $K4$ algebras to be locally finite by identifying a structure which must be present in non-locally finite modal $K4$ algebras.
We then show that this condition becomes both necessary and sufficient for $\sigma$-complete modal $K4$ algebras.
Next, we translate this condition into a pair of order-theoretic conditions on transitive Kripke frames, providing a classification of local finiteness on their dual modal algebras. 
We further show that the logic of any class of well-founded transitive relations with no infinite antichains has the finite model property, and conclude that the logic of the class of well-quasi orderings has the finite model property.
\end{abstract}

\section{Introduction}

Local finiteness emerges in many areas of mathematics, but finds a particularly natural and general expression in universal algebra and the study of varieties.
An algebra is said to be locally finite if every finitely generated subalgebra is finite. 
Local finiteness is a desirable property for algebras to have, as it implies the finite model property for its corresponding equational theory.
The logic of locally finite modal algebras are hence guaranteed to have the finite model property and be Kripke complete,
so understanding when a modal algebra is locally finite is a natural problem.
For background on universal algebra and locally finite algebras, see Burris and Sankappanavar \cite{Burris_Sankappanavar_1981}.
\par
Since the twentieth century, local tabularity of normal propositional modal logics has been studied fairly extensively.
Local tabularity is a powerful property for a logic, as it is equivalent to local finiteness of the variety of its algebras.
In particular, this property has received much consideration for normal extensions of the logic $K4$; a classical result by Segerberg \cite{Segerberg_1971} and Maksimova \cite{Maksimova_1975} shows that locally tabular extensions of $K4$ are precisely those of finite height (i.e., those containing a formula of finite height $B_h$, $h<\omega$).
Local tabularity has also been studied in various other logical contexts, including intuitionistic modal logics \cite{Bezhanishvili_Grigolia_2005} and intermediate logics \cite{Kuznetsov_1971}. 
\par
Thus, while local finiteness is fairly well understood for varieties of modal $K4$ algebras,
there appear to be far fewer characterizations of local finiteness for individual modal $K4$ algebras.
We aim to expand on this topic and provide conditions for local finiteness of such algebras. 
\par
It is a well known result that Kripke frames are not descriptive enough to represent every modal algebra.
The gaps between Kripke frames and modal algebras are bridged by general frames, which are pairs $(F,V)$ such that $F=(X,R)$ is a Kripke frame and $V$ forms a Boolean subalgebra of $\mathcal{P}(X)$, the powerset of $X$, closed under the modal operator induced by $R$.
In the semantics of general frames, valuations are required to take propositional variables to elements of $V$, so validity of modal formulas may depend on $V$ rather than solely on the underlying relation.
Consequently, general frames sit in a middle ground between Kripke frames and modal algebras, being less elementary than Kripke frames, but more intuitive than modal algebras.
For more information on general frame semantics, see \cite{Blackburn_Rijke_Venema_2001}.
\par
Every modal algebra can be adequately represented by a descriptive general frame, which is known as the dual frame of the modal algebra.
Modal algebras and their dual frames have the same logic and share many dualized properties.
Every modal algebra is embeddable in the complex algebra of a Kripke frame (Jónsson-Tarski); the full duality is due to Goldblatt and Esakia (see \cite{Blackburn_Rijke_Venema_2001}, section 5.5 for details and history).
In particular, the property of a general frame being tunable is equivalent to local finiteness of its dual modal algebra. 
This property was first considered in Franzén's proofs of Bull's theorem \cite{Segerberg_1973} and has been discussed in later works, including in connection with local tabularity; see, e.g., \cite{Shehtman_Shapirovsky_2016}.
\par
We consider modal $K4$ algebras, which are modal algebras that validate the $4$ axiom, $\Diamond \Diamond p \to \Diamond p$.
We provide a sufficient condition for local finiteness of modal $K4$ algebras. 
While this condition is not necessary in general, we show that it becomes both necessary and sufficient for $\sigma$-complete modal $K4$ algebras. 
We further translate these results into relational terms to classify the transitive Kripke frames whose associated modal algebras are locally finite. 
Using this classification, we conclude that the logic of any class of well-founded transitive relations with no infinite antichains has the finite model property.
\par
The paper is organized as follows. 
Section $2$ establishes basic definitions and preliminary results.
Section $3$ provides the algebraic criterion for local finiteness in modal $K4$-algebras, then reformulates the results into relational terms for transitive Kripke frames.
It concludes with several results illustrating the scope of these findings, including finite model property results and the resolution to some open questions posed by Shapirovsky \cite{Shapirovsky_2019} concerning tunable frames.

\section{Preliminary Definitions}

We fix some notation for binary relations. Let $X$ be a set, $R\subseteq X \times X$ a relation.
Let $\Delta_X=\{(x,x) \mid x \in X\}$ denote the diagonal of $X$. The relation $R \cup \Delta_X$ is the reflexive closure of $R$.
For $Y\subseteq X$, the restriction of $R$ to $Y$ is $R \restriction Y = \{(x,y) \in R \mid x,y \in Y\}$,
and for $x \in X$, let $R(x)=\{y \in X \mid xRy\}$.
\par
A \textit{chain} $C$ is a subset of $X$ on which $R$ restricts to a total antisymmetric transitive order. The \textit{length} of a chain is its cardinality.
We say that $R$ is \textit{well-founded} if it does not contain an infinite descending chain.
For a well-founded poset $(X,\leq)$, we define the \textit{rank} function $\rho$ from $X$ into the ordinals by setting
\[
\rho(x)=\text{sup}\{\rho(y)+1\mid y < x\}
\]
where $<$ is the strict part of $\leq$. 
Existence and uniqueness of $\rho$ follows from \cite{Jech_2002}, applied to the well-founded, irreflexive relation $<$.
\par
Throughout, a modal algebra $(V,\Diamond)$ denotes a Boolean Algebra $V$ together with a normal modal operator $\Diamond$; the Boolean operations and $\Box$ are implicitly understood.
A general frame is a pair $(F,V)$ where $F=(X,R)$ denotes a Kripke frame, and $V$ is a collection of subsets of $X$ such that $(V,\Diamond_R)$ is a modal algebra with $\Diamond_R A=R^{-1}[A]$.
Up to isomorphism, we can take $(V,\Diamond_R)$ to be its dual modal algebra; if the underlying relation $R$ is clear, we simply write $(V,\Diamond)$ as the dual algebra.
\par
Via the Jónsson-Tarski representation, every modal algebra $(V,\Diamond)$ is isomorphic to the dual modal algebra of some general frame $(X,R,V)$.
Hence it is no loss to treat the elements of $(V,\Diamond)$ as subsets of $X$ with $\subseteq$ as the order, $\cap,\cup,-$ as the Boolean operations, and $\emptyset$ as $0$.
\par
For a set $W$, let $\mathcal{P}(W)$ represent the power set of $W$. A \textit{partition} $P$ of $W$ is a set of nonempty, pairwise disjoint subsets of $W$ such that $\bigcup P =W$. 
    We say that a partition $Q$ is a \textit{refinement} of $P$ if each element of $P$ is the union of a collection of elements of $Q$. 
    \par
    Let $F=(X,R)$ be a Kripke frame. The \textit{complex algebra of $F$}, denoted $A(F)$, is the corresponding modal algebra $(\mathcal{P}(X),\Diamond_R)$, 
    where $\Diamond_R U$ denotes $R^{-1}[U]$. 
    \par
    \smallskip
    \begin{definition}
        Let $X$ be a set, $R$ a relation on $X$. A partition $P$ of $X$ is said to be \textit{$R$-tuned} if for all $A,B \in P$,
        \[
        \exists a \in A \ \exists b \in B (a R b) \implies \forall a \in A\ \exists b \in B \ (aRb).
        \]
    \end{definition}
    If the underlying relation is understood, we simply say that $P$ is tuned.
    Accordingly, a general frame $(X,R,V)$ is tunable if every finite partition $P\subseteq V$ of $X$ has a finite tuned refinement $Q \subseteq V$.
    
    Recall that an algebra $\mathbb{A}$ is \textit{locally finite} if every finitely generated subalgebra of $\mathbb{A}$ is finite.
    Typically, tunable frames are studied in the context of Kripke semantics,
    motivated by the fact that a tuned partition of a frame induces a subalgebra of its corresponding algebra; see \cite[Corollary 3.3]{Blok_1980}.
    In fact, a Kripke frame is tunable iff its algebra is locally finite.
    The corresponding statement for general frames follows similarly; see, e.g., \cite[Corollary 4.13]{Shapirovsky_2025} for details. 
    We record it here for convenience.

    \begin{prop}
        The algebra of a general frame $F$ is locally finite iff $F$ is tunable.
    \end{prop}
    Thus, to prove that a modal algebra is locally finite, it suffices to show that some (equivalently, any) dual general frame is tunable.

\section{Theorems}

\begin{definition}
    Let $(X,R,V)$ be a general frame, $Y\in V$ a subset of $X$, and $P\subseteq V$ a partition of $\bigcup P$. We define the witness function as the function $f_{Y,P}\colon Y\to \mathcal{P}(P)$ given by
    \[
    f_{Y,P}(y)=\{A \in P\mid y \in \Diamond A\}.
    \]
    Further, let $\phi_{Y,P}\colon \mathcal{P}(P)\to \mathcal{P}(Y)$ map elements of $\mathcal{P}(P)$ to their fiber under $f_{Y,P}$; that is, for $\mathcal{A}\subseteq P$,
    \[
    \phi_{Y,P}(\mathcal{A})=\{y \in Y \mid f_{Y,P}(y)=\mathcal{A}\}.
    \]
    If $\bigcup P=Y$, then we simply write $f_P$ and $\phi_P$, respectively.
\end{definition}

\begin{remark}\label{remark:f_decreasing}
    When $(X,R,V)$ is a general $K4$-frame, the function $f_{Y,P}$ is monotonically decreasing, i.e., $xRy$ implies $f_{Y,P}(y)\subseteq f_{Y,P}(x)$.
\end{remark}

\begin{proof}
    For all $y\in Y$, $y\in \Diamond A$ implies $x\in \Diamond \Diamond A$, further implying $x\in \Diamond A$ since $(X,R,V)$ is a general $K4$-frame. Hence $f_{Y,P}(y)\subseteq f_{Y,P}(x)$.
\end{proof}

\begin{lemma}\label{lemma:im(phi)_in_V}
    Let $(X,R,V)$ be a general frame, $Y\in V$ a subset of $X$, and $P\subseteq V$ a finite partition of $\bigcup P$. Then for all $\mathcal{A}\subseteq P$,
    \[
    \phi_{Y,P}(\mathcal{A})=Y\cap \bigcap_{A \in \mathcal{A}} \Diamond A \ \cap
    \bigcap_{B \in (P-\mathcal{A})}\Box (X-B)
    \]
    with the convention $\bigcap_{S \in \emptyset}(-)=X$. Consequently, $\operatorname{im}(\phi_{Y,P})\subseteq V$ and
    $\{\phi_{Y,P}(\mathcal{A})\mid \mathcal{A}\subseteq P\}-\{\emptyset\}$
    is a finite partition of $Y$ contained in $V$.
\end{lemma}

\begin{proof}
    By definition, we have $x \in \phi_{Y,P}(\mathcal{A})$ iff $x\in Y$, $x \in \Diamond A$ for all $A\in \mathcal{A}$, and $x \notin \Diamond B$ for all $B \in P-\mathcal{A}$. 
    But $x\notin \Diamond B$ iff $x\in \Box(X-B)$, and hence the above equality holds.
    Since $P$ was finite and $\mathcal{A}\subseteq P$, it follows that $\{\phi_{Y,P}(\mathcal{A})\mid \mathcal{A}\subseteq P\}-\{\emptyset\}$ is a finite set contained in $V$. That this collection partitions $Y$ follows from its elements being the fibers under $f_{Y,P}$, a function defined on $Y$.
\end{proof}

Let $\mathcal{A},\mathcal{B}$ be nonempty families of subsets of $X$ which partition $\bigcup\mathcal{A}$ and $\bigcup\mathcal{B}$, respectively. 
Define the \textit{splitting of $\mathcal{A}$ among $\mathcal{B}$} to be the collection of sets
\[
\{A \cap B \mid A \in \mathcal{A},\ B \in \mathcal{B}\}\cup
\{A-\bigcup \mathcal{B} \mid A \in \mathcal{A}\}
\]
minus any instance of the empty set. Notice that if $\mathcal{A}$ and $\mathcal{B}$ both partition $X$, then this will be the coarsest common refinement of both $\mathcal{A}$ and $\mathcal{B}$ (often denoted $\mathcal{A} \wedge \mathcal{B}$).
The following proposition is not difficult to verify.

\begin{prop} \label{prop:splitting}
    Let $(X,R,V)$ be a general frame, and $\mathcal{A},\mathcal{B} \subseteq V$ be two finite nonempty families of subsets of $X$ which partition $\bigcup \mathcal{A}$ and $\bigcup \mathcal{B}$, respectively.
    Then the splitting of $\mathcal{A}$ among $\mathcal{B}$ is a finite partition of $\bigcup \mathcal{A}$ which refines $\mathcal{A}$ and is contained in $V$.
\end{prop}

Thus, to \textit{split} $\mathcal{A}$ among $\mathcal{B}$ will mean to refine $\mathcal{A}$ into its splitting among $\mathcal{B}$.
If $\mathcal{B}=\{B_i\}_{i \in I}$ and $\mathcal{A}'$ denotes the splitting of $\mathcal{A}$ among $\mathcal{B}$, then we write $\mathcal{A} \restriction B_i$ for the collection of elements of $\mathcal{A}'$ which have nonempty intersection with $B_i$.
This technique will be quite useful in the following theorems. 
\par
Let $(X,R,V)$ be a general frame, and $\mathcal{A},\mathcal{C}\subseteq V$ be finite families of subsets of $X$ which partition $\bigcup\mathcal{A}$ and $\bigcup \mathcal{C}$ respectively.
We say $\mathcal{A}$ is \textit{tuned with respect to} $\mathcal{C}$ if 
for all $A\in \mathcal{A}$ and all $C\in \mathcal{C}$, either $A\subseteq \Diamond C$ or $A\cap \Diamond C=\emptyset$.

\begin{lemma}\label{lemma:exists_tuned_wrt}
    Let $(X,R,V)$ be a general frame 
    with $\mathcal{A},\mathcal{C} \subseteq V$ finite nonempty families of subsets of $X$ which partition $\bigcup\mathcal{A}$ and $\bigcup \mathcal{C}$ respectively.
    Then there is a finite refinement $\mathcal{A}'\subseteq V$ of $\mathcal{A}$ which is tuned with respect to $\mathcal{C}$.
\end{lemma}

\begin{proof}
    For each $x \in \bigcup \mathcal{A}$, there is a collection $\mathcal{W}\subseteq\mathcal{C}$ such that $x \in \Diamond W$ for each $W\in \mathcal{W}$ and $x\notin \Diamond U$ for any $U \in \mathcal{C}-\mathcal{W}$.
    Split $\mathcal{A}$ among the (nonempty) family of disjoint subsets
    \[
    \{\phi_{\bigcup \mathcal{A}, \mathcal{C}}(\mathcal{W}) \mid \mathcal{W} \subseteq \mathcal{C}\}-\{\emptyset\}
    \]
    and call the result of this splitting $\mathcal{A}'$. Since the above family is a finite, nonempty collection of members of $V$, from Proposition \ref{prop:splitting} it follows that $\mathcal{A}' \subseteq V$ and is finite.
    Now, let $A \in \mathcal{A}'$, $C \in \mathcal{C}$, and suppose $A\cap \Diamond C\neq \emptyset$. Then there exists a nonempty $\mathcal{W}\subseteq \mathcal{C}$ such that $C \in \mathcal{W}$ and
    \[
    A \subseteq \phi_{\bigcup \mathcal{A},\mathcal{C}}(\mathcal{W})=\bigcup \mathcal{A} \cap  \bigcap_{W \in \mathcal{W}}\Diamond W \cap \bigcap_{U\in (\mathcal{C}-\mathcal{W})}\Box (X-U)
    \]
    and hence $A  \subseteq \Diamond C$, as desired.
\end{proof}

It is worth noting that if $\mathcal{A}$ is tuned with respect to $\mathcal{C}$ and every member of $\mathcal{B}$ is contained in a member of $\mathcal{A}$, then $\mathcal{B}$ is tuned with respect to $\mathcal{C}$ as well.
Moreover, a partition $P\subseteq V$ of $X$ is tuned precisely when $P$ is tuned with respect to itself.

\begin{definition} \label{def:diamond_tower}
    We say that a modal algebra $(V,\Diamond)$ contains an infinite descending $\Diamond$-tower if there exists a (countable) sequence of nonempty elements $(T_i)_{i\in \mathbb{Z}_{>0}}\subseteq V$ satisfying the following properties for all $i \in \mathbb{Z}_{>0}$:
    \begin{itemize}
        \item [(i)] $T_{i+1} \subseteq \Diamond T_i$
        \item [(ii)] $\Diamond T_{i+1} \cap T_{i}=\emptyset$
        \item [(iii)] $T_{i+1}\cap T_{i}=\emptyset$
    \end{itemize}
\end{definition}

$(T_i)_i$ is accordingly referred to as an infinite descending $\Diamond$-tower.
The 'tower' aspect of this construction arises from its translation into a relational setting. Informally, $T_1$ forms the 'top' of the tower, with $T_2$ lying entirely underneath $T_1$, $T_3$ lying entirely underneath $T_2$, and so on. 
The following lemma is not difficult to verify, and provides a useful description of these $\Diamond$-towers in the $K4$ setting.
\begin{prop}\label{prop:transitive_tower}
    Let $(V,\Diamond)$ be a modal $K4$ algebra containing an infinite descending $\Diamond$-tower $(T_i)_i$.
    Then for all $i \in \mathbb{Z}_{>0}$ and all $j>i$,
    \begin{itemize}
        \item [(i)] $T_j \subseteq \Diamond T_i$
        \item [(ii)] $\Diamond T_j \cap T_i=\emptyset$
        \item [(iii)] $T_j \cap T_i=\emptyset$
    \end{itemize}
\end{prop}

\par
The following theorem provides a sufficient condition for local finiteness of modal $K4$ algebras.

\begin{theorem} \label{thm:no_tower_implies_locally_finite}
    Let $(V,\Diamond)$ be a modal $K4$ algebra.
    Then if $(V,\Diamond)$ does not contain an infinite descending $\Diamond$-tower, it is locally finite.
\end{theorem}

\begin{proof}
    Let $(X,R,V)$ be any general frame whose dual modal algebra is $(V,\Diamond)$, and assume $(V,\Diamond)$ has no infinite descending $\Diamond$-towers. 
    Let $P_1 \subseteq V$ be a finite partition of $X$.
    We construct a refinement process of $P_1$ with the following properties: if this process does not terminate after finitely many steps, then it produces an infinite descending $\Diamond$-tower; if this process terminates after finitely many steps, then it produces a finite tuned refinement of $P_1$ contained in $V$, hence implying that the subalgebra generated by $P_1$ is finite.
    \par \textbf{Step 1}:
    Let
    \[
    T_1= \bigcup_{\mathcal{A}\subseteq P_1}(\phi_{P_1}(\mathcal{A}) \cap \Box \phi_{P_1}(\mathcal{A})).
    \]
    \textbf{Claim.} The following hold:
    \begin{itemize}
    \item [(i)] $T_1 \subseteq \Box T_1$
    \item [(ii)] $X-T_1 \subseteq \Diamond T_1$.
    \end{itemize}

    \begin{proof}
    \begin{itemize}
    
  \item[(i)] 
  If $x\in T_1$, then $x \in \phi_{P_1}(\mathcal{A})\cap\Box \phi_{P_1}(\mathcal{A})$ for some $\mathcal{A}\subseteq P_1$, implying 
  \[
  x \in \Box \phi_{P_1}(\mathcal{A})\cap \Box \Box \phi_{P_1}(\mathcal{A})=\Box(\phi_{P_1}(\mathcal{A})\cap\Box \phi_{P_1}(\mathcal{A}))\subseteq \Box T_1.
  \]
  
  \item[(ii)]  
    Let $x_1\in X-T_1$. Clearly, $x_1\in \phi_{P_1}(f_{P_1}(x_1))$, and so $x_1 \notin \Box \phi_{P_1}(f_{P_1}(x_1))$. 
    Thus, there is some $x_2\in R(x_1)$ with $x_2\notin \phi_{P_1}(f_{P_1}(x_1))$, implying $f_{P_1}(x_2)\neq f_{P_1}(x_1)$. 
    By Remark \ref{remark:f_decreasing}, $f_{P_1}(x_2)\subseteq f_{P_1}(x_1)$, further giving us $f_{P_1}(x_2)\subsetneq f_{P_1}(x_1)$.
    If $x_2\notin T_1$, then we likewise find some $x_3 \in R(x_2)$ with $f_{P_1}(x_3)\subsetneq f_{P_1}(x_2)$.
    Continue in this manner to obtain some $x_n\in T_1$ $(n\geq 2)$ with $x_1Rx_2R\dots Rx_n$, which exists by finiteness of $\mathcal{P}(P_1)$. 
    Thus, $x_1\in \Diamond ^{n-1}T_1\subseteq \Diamond T_1$, as desired.
    \end{itemize}
    \end{proof}

    Now, if $X-T_1$ is nonempty, we split $P_1$ among $T_1$ and $X-T_1$. Further split $P_1 \restriction T_1$ among 
    \[
    \{\phi_{P_1}(\mathcal{A})\mid \mathcal{A} \subseteq P_1\}-\{\emptyset\}
    \]
    which is a finite partition of $X$ contained in $V$, by Lemma \ref{lemma:im(phi)_in_V}. Let $F_1$ denote the result of this splitting.
    \par
    The partition $F_1$ is tuned over $(T_1, R\restriction T_1)$.
        To see why, let $A,B \in F_1$ such that there is an $a \in A$ and $b \in B$ with $aRb$. By the above splitting, we have 
        \[
        A=A'\cap \phi_{P_1}(\mathcal{A})\cap T_1
        \]
        and
        \[
        B=B'\cap \phi_{P_1}(\mathcal{B})\cap T_1
        \]
        for some $A',B'\in P_1$ and some $\mathcal{A},\mathcal{B}\subseteq P_1$.
        Since $a \in \Box \phi_{P_1}(\mathcal{A})$, it follows that $b \in \phi_{P_1}(\mathcal{A})$, i.e., $\phi_{P_1}(\mathcal{A})=\phi_{P_1}(\mathcal{B})$.
        Now, suppose $c\in A$. Then $c \in A'\cap \phi_{P_1}(\mathcal{A})\cap T_1$, and hence $c \in \Diamond B'$ as well. 
        But $c \in \Box\phi_{P_1}(\mathcal{A})$ (since $c \in T_1)$ and $c \in \Box T_1$ (by (i) of the above claim), implying
        \begin{align*}
            c\in \Diamond(B'\cap \phi_{P_1}(\mathcal{A})\cap T_1) &=\Diamond(B'\cap \phi_{P_1}(\mathcal{B})\cap T_1)\\ &=\Diamond B
        \end{align*}
        that is, $cRd$ for some $d \in B$, as desired.
    \par
    Thus, if $X=T_1$, then we are done and the process ends. If not, let $X_2=X-T_1$. 
    By Lemma \ref{lemma:exists_tuned_wrt}, we obtain a finite refinement $P_2\subseteq V$ of $P_1\restriction X_2$ which is tuned with respect to $F_1$.
    This concludes step 1.
    \par
    \textbf{Step n+1}: Given that we have obtained $X_{n+1}=X-\bigcup_{i\leq n}T_i$ and $P_{n+1}$ from step $n$, we repeat this process with respect to these sets.
    This step will be similar to step 1 in design, but will contain a number of finer differences that warrants an explicit description.
    Define
    \[
    T_{n+1}=\bigcup_{\mathcal{A}\subseteq P_{n+1}}
    \phi_{P_{n+1}}(\mathcal{A}) \cap \Box(\bigcup_{i\leq n}T_i \cup\phi_{P_{n+1}}(\mathcal{A})).
    \]
    \textbf{Claim.} 
    \begin{itemize}
    \item[(i)] $T_{n+1}\subseteq \Box\bigcup_{i\leq n+1}T_i$
    \item[(ii)] $X-\bigcup_{i\leq n+1}T_i \subseteq \Diamond T_{n+1}$.
    \end{itemize}
    
    \begin{proof}
    \begin{itemize}
    \item[(i)] Let $x \in T_{n+1}$.
    Suppose  $xRy$, and further suppose $y \notin \bigcup_{i\leq n}T_i$. We show that $y\in T_{n+1}$.
    Indeed, since $x \in T_{n+1}$, then $x \in \phi_{P_{n+1}}(\mathcal{A}) \cap \Box(\bigcup_{i\leq n}T_i \cup\phi_{P_{n+1}}(\mathcal{A}))$ for some $\mathcal{A}\subseteq P_{n+1}$. 
    From this, we further obtain 
    \[
    x \in \Box(\bigcup_{i\leq n}T_i \cup\phi_{P_{n+1}}(\mathcal{A})) \cap \Box \Box(\bigcup_{i\leq n}T_i \cup\phi_{P_{n+1}}(\mathcal{A}))
    \]
    implying
    \[
    y \in (\bigcup_{i\leq n}T_i \cup\phi_{P_{n+1}}(\mathcal{A})) \cap \Box(\bigcup_{i\leq n}T_i \cup\phi_{P_{n+1}}(\mathcal{A}))
    \]
    But $y \notin \bigcup_{i\leq n}T_i$, giving us
    \[
    y \in \phi_{P_{n+1}}(\mathcal{A}) \cap \Box(\bigcup_{i\leq n}T_i \cup\phi_{P_{n+1}}(\mathcal{A}))\subseteq T_{n+1}
    \]
    as desired.
    
    \item[(ii)] Let $x_1\in X-\bigcup_{i\leq n+1}T_i$. Clearly, $x_1\in \phi_{P_{n+1}}(f_{P_{n+1}}(x_1))$, and so $x_1 \notin \Box(\bigcup_{i\leq n}T_i \cup \phi_{P_{n+1}}(f_{P_{n+1}}(x_1)))$. 
    Thus, there is some $x_2\in R(x_1)\cap X_{n+1}$ with $x_2\notin \phi_{P_{n+1}}(f_{P_{n+1}}(x_1))$, implying $f_{P_{n+1}}(x_2)\neq f_{P_{n+1}}(x_1)$. 
    By Remark \ref{remark:f_decreasing}, $f_{P_{n+1}}(x_2)\subseteq f_{P_{n+1}}(x_1)$, further giving us $f_{P_{n+1}}(x_2)\subsetneq f_{P_{n+1}}(x_1)$.
    If $x_2\notin T_{n+1}$, then we likewise find some $x_3 \in R(x_2)\cap X_{n+1}$ with $f_{P_{n+1}}(x_3)\subsetneq f_{P_{n+1}}(x_2)$.
    Continue in this manner to obtain some $x_m\in T_{n+1}$ $(m\geq 2)$ with $x_1Rx_2R\dots Rx_m$, which exists by finiteness of $\mathcal{P}(P_{n+1})$. 
    Thus, $x_1\in \Diamond ^{m-1}T_{n+1}\subseteq \Diamond T_{n+1}$, as desired.
    \end{itemize}
    \end{proof}
    \par
    If $X_{n+1}-T_{n+1}$ is nonempty, we split $P_{n+1}$ among $T_{n+1}$ and $X_{n+1}-T_{n+1}$. Further split $P_{n+1}\restriction T_{n+1}$ among the family
    \[
    \{\phi_{P_{n+1}}(\mathcal{A})\mid \mathcal{A}\subseteq P_{n+1}\} - \{\emptyset\}
    \]
    which is a finite family of pairwise disjoint members of $V$ by Lemma \ref{lemma:im(phi)_in_V}. Call the result of this splitting $F_{n+1}$. 
    By similar reasoning used to prove that $F_1$ was tuned, $F_{n+1}$ is a tuned partition of $(T_{n+1},R\restriction T_{n+1})$ contained in $V$.
    \par
    Now, if $X_{n+1}=T_{n+1}$, then we still obtain a finite tuned partition $F_{n+1}$ of $(T_{n+1}, R\restriction T_{n+1})$ via the same reasoning.
    Moreover, each $F_j$ is a finite tuned partition of $T_j$ and is tuned with respect to itself via the $\phi_{P_j}$-splittings, with respect to $F_i$ for $i>j$ (vacuously, since $A\cap \Diamond B=\emptyset$ for all $A \in F_j,\ B\in F_i$), and with respect to $F_i$ for $i<j$ (by our construction).
    It follows that $\bigcup_{i=1}^{n+1}F_i$ is a finite tuned refinement of $P_1$ contained in $V$, so the subalgebra generated by $P_1$ is finite. In this case, the process terminates.
    \par
    On the other hand, if $X_{n+1}\neq T_{n+1}$, then let $X_{n+2}=X_{n+1}-T_{n+1}$ and use Lemma \ref{lemma:exists_tuned_wrt} to obtain a finite refinement $P_{n+2}\subseteq V$ of $P_{n+1}\restriction X_{n+2}$ which is tuned with respect to $F_{n+1}$.
    This concludes step n+1.
    \par
    If this process does not terminate, then each $T_j$ is nonempty. 
    Indeed, if some $T_j$ was empty, then we would obtain 
    \[
    X-\bigcup_{i\leq j}T_i =X_{j}\subseteq \Diamond T_{j}=\Diamond \emptyset=\emptyset
    \]
    i.e., $X_j=\emptyset$, contrary to our assumption that the process hasn't terminated.
    Moreover, $(T_i)_i$ is an infinite descending $\Diamond$-tower: condition (i) follows from claim (ii) in step $i$, condition (ii) follows from claim (i) in step $i$, and condition (iii) follows from construction.
    Thus, this process terminates after finitely many steps, implying that it has a finite tuned refinement contained in $V$.
    Since $P_1$ was an arbitrary finite partition of $X$ contained in $V$, it follows that $(X,R,V)$ is tunable and $(V,\Diamond)$ is locally finite.    
\end{proof}

\begin{remark}
    The converse is not true in general. 
    Consider the general frame $(\omega,\geq,V)$, where $\geq$ is the reverse ordering on the natural numbers and $V$ consists of all finite and cofinite subsets of $(\omega,\geq)$ (with cofinite subsets being the complements of finite subsets).
    It is not too difficult to check that $(V,\Diamond)$ is a modal algebra.
    \par
    Notice that the family of singletons form an infinite descending $\Diamond$-tower in $V$. Nevertheless, $V$ is locally finite:
    indeed, any finite partition of $\omega$ contained in $V$ will be of the form $\{C,F_1,\dots,F_n\}$, where $C$ is cofinite and $F_j$ is finite, 
    and for $N=\text{inf}(\bigcup_i F_i)$, the partition $\{\{0\},\dots,\{N\},\{m\mid m> N\}\}$ is a finite tuned refinement of $\{C,F_1,\dots,F_n\}$.
\end{remark}
The converse does hold, however, for modal $K4$ algebras which are closed under countable join. Such algebras are known as $\sigma$-complete modal $K4$ algebras. 
Note that countable joins in $V$ need not coincide with countable unions, so the following proof will instead utilize Boolean operations and $\leq$ as the order rather than set-theoretic operations and $\subseteq$.

\begin{prop}\label{prop:complete_implies_tower_implies_not_locally_finite}
    Let $(V,\Diamond)$ be a $\sigma$-complete modal $K4$ algebra. If $V$ contains an infinite descending $\Diamond$-tower, then it is not locally finite.
\end{prop}

\begin{proof}
    Let $(T_j)_j\subseteq V$ be this infinite descending $\Diamond$-tower. Further let
    \[
    \mathcal{O}=\bigvee_{j\text{ odd}}T_j
    \] 
    and
    \[
    \mathcal{E}=\bigvee_{j\text{ even}}T_j
    \]
    which are both members of $V$ since $V$ is $\sigma$-complete.
    We show 
    \[
    \Diamond \mathcal{E} \wedge \mathcal{O}=\bigvee_{j \text{ odd},\ j\geq 3, }T_j.
    \]
    The $\geq$ direction follows from the definition of $\Diamond$-tower.
    For the $\leq$ direction, observe that
    $\mathcal{E}\leq \Diamond T_2\vee T_2$, and so 
    $\Diamond \mathcal{E}\leq\Diamond(\Diamond T_2 \vee T_2)\leq \Diamond T_2$.
    Hence, since $\Diamond T_2 \wedge T_1 =0$, it follows that $\Diamond\mathcal{E} \wedge T_1=0$, and thus 
    \begin{align*}
    \Diamond \mathcal{E}\wedge \mathcal{O}=\Diamond\mathcal{E} \wedge \bigvee_{j \text{ odd}}T_j&=\Diamond \mathcal{E} \wedge (T_1 \vee \bigvee_{j \text{ odd}, \ j\geq 3}T_j)\\
    &=(\Diamond \mathcal{E} \wedge T_1)\vee (\Diamond\mathcal{E}\wedge \bigvee_{j \text{ odd}, \ j\geq 3}T_j) \\
    &= \Diamond\mathcal{E}\wedge \bigvee_{j \text{ odd}, \ j\geq 3}T_j\\
    &\leq \bigvee_{j \text{ odd}, \ j\geq 3}T_j.
    \end{align*}
    Hence, $\Diamond\mathcal{E}\wedge \mathcal{O}=\bigvee_{j\text{ odd},\ j\geq 3}T_j\in \langle\mathcal{E},\mathcal{O}\rangle$.
    \par
    Now, for $T'_n=T_{n+1}$, we see that $\{T_1',T_2',\dots\}$ forms an infinite descending $\Diamond$-tower itself, with $\mathcal{E}$ being the join of its odd indices and $\Diamond \mathcal{E\wedge O}$ being the join of its even indices.
    Thus, by a simple induction argument (using the argument already given), $\bigvee_{j\text{ odd, }j\geq n} T_j \in \langle\mathcal{E},\mathcal{O}\rangle$ for each odd positive integer $n$. But each $\bigvee_{j\text{ odd, }j\geq n} T_j$ is distinct, so $\langle\mathcal{E},\mathcal{O}\rangle$ is infinite, as desired.
\end{proof}

Combining this with Theorem \ref{thm:no_tower_implies_locally_finite} gives us the following.

\begin{corollary} \label{coro:classification_of_lfcmk4a}
    Let $(V,\Diamond)$ be a $\sigma$-complete modal $K4$ algebra. Then $(V,\Diamond)$ is locally finite if and only if it does not contain an infinite descending $\Diamond$-tower.
\end{corollary}

With this classification of locally finite $\sigma$-complete modal $K4$ algebras, it is reasonable to ask how it translates to a classification of transitive Kripke frames whose dual modal algebras are locally finite. We show that local finiteness of these algebras is equivalent to a comparatively simple pair of properties holding in their respective frames.
\par
Recall that the \textit{disjoint union} of a family of frames $(F_i)_{i \in I}$ where $F_i=(X_i,R_i)$ is the frame 
    $\bigsqcup_{i\in I}F_i = \left( \bigsqcup_{i \in I}X_i\ ,R_I \right)$,
    such that 
    $\bigsqcup_{i \in I}X_i = \bigcup_{i \in I} (\{i\} \times X_i)$, and
    \[
    (i,x) R_I (j,y)\ \ \ \ \text{ iff } \ \ \ \ (i=j \ \text{ and }\ xR_i y).
    \]
Let $D$ denote the disjoint union of all finite ordinals viewed as reflexive linear orders under their usual order $\leq$ ; that is, 
$D=\bigsqcup_{\alpha \in \omega} (\alpha, \leq_\alpha)$.
\begin{definition}
    A transitive frame $(X,R)$ is said to be \textit{$D$-free} if no subframe of $(X,R\cup \Delta_X)$ is isomorphic to $D$.
\end{definition}
By a subframe of $(X,R)$, we mean a frame of the form $(Y,R\restriction Y)$ where $Y \subseteq X$.
The notion of $D$ will be considerably important in the following section, so here we provide some intuition for its choosing via the proposition below.
A transitive frame has \textit{finite height} if there is a finite bound on the length of chains in the frame, and has infinite height otherwise.

\begin{prop}\label{prop:fin_height}
    Let $(X,R)$ be a transitive frame. $(X,R)$ has finite height if and only if $R$ is well-founded, Noetherian (converse well-founded), and $D$-free.
\end{prop}

We prove this after proving Theorem \ref{thm:classification_of_tunable_transitive_frames}, as it will greatly simplify the proof.
Recall that any transitive frame whose logic is locally tabular must have finite height, and thus is well-founded, Noetherian, and $D$-free. Informally, locally tabular normal extensions of $K4$ reject infinite chains and the presence of $D$.
As will be shown in the following section, local finiteness of (complex) modal $K4$ algebras is strictly more general: their dual Kripke frames accept infinite ascending chains, but still reject infinite descending chains and the presence of $D$.
This is due to infinite descending $\Diamond$-towers: the existence of such a tower in a complex modal $K4$ algebra is equivalent to the dual Kripke frame either not being well-founded or not being $D$-free.
\par
On a well-founded poset $(X,\leq)$, recall the rank function $\rho$ from $X$ into the ordinals:
\[
\rho(x)=\text{sup}\{\rho(y)+1\mid y < x\}.
\] 
We define $M_1$ to be the set of minimal elements of $(X,\leq)$, and inductively define $M_{n+1}$ to be the set of minimal elements of the subposet $X-\bigcup_{j=1}^nM_j$. 
The following lemma, which is straightforward to verify, connects these constructions.
\begin{lemma}\label{claim1}
    For a well-founded poset $(X,\leq)$, the following are equivalent:
    \begin{itemize}
        \item [(a)] $x \in M_n$
        \item [(b)] $\rho(x)=n-1$
        \item [(c)] The longest chain for which $x$ is the maximum has length $n$.
    \end{itemize}
    Thus, if $z_1<\dots<z_n=x$ is a chain of length $n$ with maximum $x \in M_n$, then $z_i \in M_i$ for each $i$.
\end{lemma}

\begin{lemma}\label{claim2}
    Let $(X,\leq)$ be a well-founded poset whose complex algebra contains an infinite descending $\Diamond$-tower $(T_j)_j$ which partitions $X$. 
    Then for each $n \in \mathbb{Z}_{>0}$, there are infinitely many $T_j$'s such that $M_n\cap T_j$ is nonempty.
\end{lemma}

\begin{proof}
    By way of contradiction, suppose not. Then there is some $M_n$ which has nonempty intersection with only finitely many $T_j$'s. Then since $\{T_1,T_2,\dots\}$ partitions $X$, there is some $m$ with $M_n \subseteq \bigcup_{j=1}^m T_j$.
    Let $Y=(X-\bigcup_{i=1}^{n-1} M_i)-\bigcup_{j=1}^mT_j$, and observe that it has no minimal elements, since any minimal element of 
    $Y$ is minimal for $X-\bigcup_{i=1}^{n-1} M_i$ as well (in particular, by (ii) in Proposition \ref{prop:transitive_tower}). 
    However, notice that for each $k\in \mathbb{Z}_{>0}$, there is an element $x\in T_{k}$ which is the maximum of a chain of length $n$.
    This $x$ can be found by picking an arbitrary $y \in T_{k+n-1}$ and using (i) in Proposition \ref{prop:transitive_tower} to build a chain of length $n$ with least element $y$ and greatest element $x$.
    Thus $\rho(x)\geq n-1$, implying by Lemma \ref{claim1} that $x \notin\bigcup_{i=1}^{n-1}M_i$. 
    But then $Y \cap T_k$ is nonempty for each $k\geq m+1$, and so $Y$ is an (infinite) subset of $X$ with no minimal elements, contradicting that $X$ is well-founded.
\end{proof}

\begin{lemma}\label{lemma:tower&wf_implies_Dfree}
    Let $F=(X,\leq)$ be a poset. Then if $A(F)$ admits an infinite descending $\Diamond$-tower, $F$ is either not well-founded or not $D$-free.
\end{lemma}

\begin{proof}
    Assume $F$ is well-founded, and 
    let $(T_i)_{i\in \mathbb{Z}_{>0}}$ denote the infinite descending $\Diamond$-tower contained in $A(F)$. 
    By restricting to $\bigcup_{i}T_i$, without loss of generality, we can assume $\{T_1,T_2,\dots\}$ partitions $X$.
    \par
    We now construct $D$ as a subframe of $F$ via an infinite-step process. 
    The idea of the construction is the following. At step $n$, we select a point which lies in some $M_i$ and $T_j$ for sufficiently large $i,j$,
    so that it lies outside all previously chosen structures. 
    Using Lemma \ref{claim1}, we then extract a chain of length $n$ whose top segment will serve as the copy of $(n,\leq)$. 
    Finally, we remove some collection of $M_i$'s and $T_j$'s so that the chain obtained in the next step neither witnesses nor is witnessed by any of the chains obtained in previous steps.
    \par
    Let $(m_n)_{n \in \mathbb{Z}_{>0}}$ be any strictly increasing sequence of natural numbers such that $m_{n+1}> m_n+n$. 
    It is worth noting that each $M_n$ is nonempty by Lemma $\ref{claim2}$.
    The process is defined as follows:
    \par \smallskip
    \textbf{Step 1}: Pick any $x\in M_{m_1}$, and say $x \in T_{k_1}$. Let
    \[
    F_1=X-(\bigcup_{i=1}^{m_1} M_i \cup \bigcup_{j=1}^{k_1} T_j).
    \]
    Further let $Z_{1}=\{x\}$, and note that $Z_1\subseteq \bigcup_{i=1}^{m_1}M_i \cap \bigcup_{j=1}^{k_1}T_j$ and is clearly isomorphic to $(1,\leq)$.
    \par \smallskip
    \textbf{Step n+1}: Assume inductively that after step n we have constructed $Z_n$ and $F_n$ such that 
    \[
    F_n=X-(\bigcup_{i=1}^{m_n} M_i \cup \bigcup_{j=1}^{k_n} T_j), \ \ Z_n \subseteq \bigcup_{i=1}^{m_n} M_i \cap \bigcup_{j=1}^{k_n}T_j
    \] 
    where $Z_n \cong \bigsqcup_{i\leq n}(i,\leq_i)$.
    Since only finitely many $T_j$'s were removed in the definition of $F_n$ and each $M_j$ has nonempty intersection with infinitely many $T_i$'s (Lemma \ref{claim2}), it follows that $M_j \cap F_n$ is nonempty for all $j>m_n$.
    Hence, choose some $z_{n+1} \in F_n \cap M_{m_{n+1}}$. 
    Then by Lemma \ref{claim1}, $z_{n+1}$ is the maximum of a chain of length $m_{n+1}$. Let
    \[
    z_{n+1} \geq z_{n}\geq \dots \geq z_{2} \geq z_{1}
    \]
    be the top $n+1$ elements of this chain, which we call $C$. 
    Lemma \ref{claim1} also gives us that, for each $z_j\in C$,
    \[
    z_{j} \in M_{m_{n+1}+j-(n+1)}
    \]
    and hence, since $m_{n+1}>m_n + n$,
    \[
    m_{n+1}+j-(n+1)> m_n +n+j-(n+1)=m_n+j-1\geq m_n
    \]
    implying that $z_{j} \in X-\bigcup_{i=1}^{m_n}M_i$.
    Now, let $w \in Z_n$. 
    Since $Z_n \subseteq \bigcup_{i=1}^{m_n}M_i$, it follows that $\rho(z_j)\geq m_n>\rho(w)$, and hence $z_j\nleq w$.
    Moreover, $z_{j} \leq z_{n+1}\in X-\bigcup_{i=1}^{k_n}T_i$, and 
    since $Z_n\subseteq\bigcup_{i=1}^{k_n}T_i$,
    it follows from (ii) in Proposition \ref{prop:transitive_tower} that $w \nleq z_j$.
    Thus $Z_{n+1}=Z_n \cup C$ is isomorphic (as relational structures) to $\bigsqcup_{i\leq n+1}(i,\leq_i)$.
    \par
    To finish, suppose that $z_{1} \in T_{k_{n+1}}$, and noting that $Z_{n+1}\subseteq \bigcup_{i=1}^{m_{n+1}}M_i \cap \bigcup_{j=1}^{k_{n+1}}T_j$, let
    \[
    F_{n+1}=X-(\bigcup_{i=1}^{m_{n+1}} M_i \cup \bigcup_{j=1}^{k_{n+1}} T_j).
    \]
    This concludes step n+1.
    \par
    \smallskip
    As this process will not end, we see that $Z=\bigcup_{i \in \mathbb{Z}_{>0}} Z_i$ is a subframe of $F$ isomorphic to $D$, as desired.
\end{proof}

\begin{lemma} \label{lemma:D&idc_not_tunable}
    The complex modal algebras of both $D$ and $(\omega,\geq)$, the natural numbers with the reverse ordering, contain an infinite descending $\Diamond$-tower.
\end{lemma}

\begin{proof}
For $(\omega,\geq)$, the singletons of $\mathcal{P}(\omega)$ precisely form the tower.
For $D$, recall that its elements are of the form $(i,x)$, with $x < i$, ordered so that $(i,x)\leq (j,y)$ iff $i=j$ and $x\leq y$. Define
\[
D_k=\{(i,\ i-k)\mid i\geq k\}.
\]
Thus $D_k$ consists exactly of those elements lying $k-1$ steps below the top of their respective chain.
It is not difficult to check that $(D_k)_{k\geq 1}$ forms an infinite descending $\Diamond$-tower.
\end{proof}

Recall that a quasi-order (preorder) is a relation that is reflexive and transitive.
A \textit{cluster} is a frame of the form $(X,X\times X)$ for a set $X$.
The \textit{skeleton} of a quasi-order $F=(X,R)$ is the partial order sk$F=(\overline{X},\leq_R)$,
where $\overline{X}$ is the resulting quotient of $X$ by the equivalence $R \cap R^{-1}$, and ${A}\leq_R{B}$ iff there is some $a \in A$ and $b\in B$ with $aRb$.

\begin{lemma} \label{lemma:quasi-ordering_classification}
    Let $F=(X,R)$ be a quasi-ordering. Then $F$ is well-founded and $D$-free if and only if $A(F)$ admits no infinite descending $\Diamond$-towers. 
\end{lemma}

\begin{proof}
    For the forward direction, suppose contrapositively that $A(F)$ admits an infinite descending $\Diamond$-tower $\{T_1,T_2,\dots\}$. Notice that if $i\neq j$, then we cannot have $T_i\cap A\neq \emptyset$ and $T_j \cap A\neq \emptyset$ for any $A\in \overline{X}$, as that would violate (ii) in Proposition \ref{prop:transitive_tower}. 
    Hence, for $T_j'=\{A \in \overline{X} \mid T_j \cap A \neq \emptyset\}$, the set $\{T_1',T_2',\dots\}$ is an infinite descending $\Diamond$-tower in $A(\text{sk}(F))$.
    But sk$(F)$ is a partial ordering, and hence is either not well-founded or not $D$-free by Lemma \ref{lemma:tower&wf_implies_Dfree}.
    It is easy to verify that $F$ is consequently either not well-founded or not $D$-free.
    For the converse direction, observe that if $F$ had a subframe relationally isomorphic to $D$ or to $(\omega,\geq)$, then $A(F)$ inherits an infinite descending $\Diamond$-tower from Lemma \ref{lemma:D&idc_not_tunable}.
\end{proof}

\begin{lemma} \label{lemma:reflexivity_doesn't_matter}
    Let $(X,R)$ be a frame. Then $A(X,R)$ contains an infinite descending $\Diamond$-tower if and only if $A(X,R\cup \Delta_X)$ contains an infinite descending $\Diamond$-tower.
\end{lemma}

\begin{proof}
    Suppose $A(X,R)$ contains an infinite descending $\Diamond$-tower $\{T_1,T_2,\dots\}$.
    Then viewing $\{T_1,T_2,\dots\}$ as a collection of nonempty, pairwise disjoint subsets of $X$, one can easily check that the addition of reflexive loops will not change that $\{T_1,T_2,\dots\}$ is an infinite descending $\Diamond$-tower in $(X,R\cup \Delta_X)$. 
    The converse direction, involving removing some reflexive loops, is similarly easily verifiable.
\end{proof}

Hence, we obtain the following classification:

\begin{theorem} \label{thm:classification_of_tunable_transitive_frames}
    Let $F$ be a transitive frame. Then $A(F)$ is locally finite if and only if $F$ is well-founded and $D$-free.
\end{theorem}

\begin{proof}
    Let $F=(X,R)$. Then $(X,R)$ is a well-founded and $D$-free transitive frame iff $(X,R\cup \Delta_X)$ is well-founded and $D$-free quasi-order, which is equivalent to $A(X,R\cup \Delta_X)$ containing no infinite descending $\Diamond$-towers by Lemma \ref{lemma:quasi-ordering_classification}.
    But by Lemma \ref{lemma:reflexivity_doesn't_matter}, this is equivalent to $A(X,R)$ having no such $\Diamond$-tower, which is further equivalent to $A(X,R)$ being locally finite by Corollary \ref{coro:classification_of_lfcmk4a}.
\end{proof}

It is a standard observation that the logic of an algebra is equal to the logic of the class of its finitely generated subalgebras. Hence, if an algebra is locally finite, then its logic is complete with respect to a class of finite algebras (whose duals are finite Kripke frames), and thus has the finite model property.
The following corollary is an immediate consequence of this fact and the above theorem.

\begin{corollary}
    The logic of any class of well-founded transitive frames with no infinite antichains has the finite model property. 
    In particular, the logic of any class of well-quasi orderings (well-founded quasi-orderings with no infinite antichains) has the finite model property.
\end{corollary}

Additionally, we can now prove Proposition \ref{prop:fin_height}, as promised. For a transitive frame $(X,R)$, we say $x\in X$ is maximal if  for each $y\in X$, $xRy$ implies $yRx$.

\begin{proof}[Proof of Proposition \ref{prop:fin_height}.]
    The forward direction is clear, as transitive frames with finite height are certainly well-founded, Noetherian, and $D$-free.
    Conversely, suppose by contrapositivity that $(X,R)$ has infinite height, and furthermore, is Noetherian. We show that $(X,R)$ is either not well-founded or not $D$-free. 
    Define $N_1$ to be the set of maximal elements of $X$, and inductively define $N_k$, $k\geq 2$, to be the set of maximal elements of the subframe $X-\bigcup_{i=1}^{k-1}N_i$.
    Since $(X,R)$ has infinite height, each $X-\bigcup_{i=1}^{n-1}N_i$ is nonempty and has infinite height, and hence, since $(X,R)$ is Noetherian, each $N_k$ is nonempty.
    It is not difficult to check that $N_1,N_2,\dots$ forms an infinite descending $\Diamond$-tower, which implies $(X,R)$ is either not well-founded or not $D$-free by Theorem \ref{thm:classification_of_tunable_transitive_frames}.
\end{proof}

Several open questions were raised by Shapirovsky in \cite{Shapirovsky_2019} regarding tunable Kripke frames, which can now be answered via Theorem \ref{thm:classification_of_tunable_transitive_frames}.
One question was whether the direct product of two tunable frames is a tunable frame in general.
In fact, this is not the case.
Consider the frames $(A, \Delta_A)$ and $(\omega, \leq)$, where $A$ is countably infinite and $(\omega, \leq)$ is the standard ordering on the natural numbers.
Both are tunable frames, as they are well-founded and $D$-free.
However, their direct product is isomorphic to the disjoint union of countably many copies of $(\omega,\leq)$, which is not $D$-free, and hence not tunable.
\par
For a finite family $(\alpha_i)_{i\in I}$ of ordinals, it was also asked whether the algebras of the direct products $\prod_{i \in I}(\alpha_i,\leq)$, $\prod_{i \in I}(\alpha_i,<)$ are locally finite.
A standard generalization of Dickson's Lemma \cite{Dickson_1913} states that the finite direct product of well-quasi orderings is a well-quasi ordering. 
Hence, by Theorem \ref{thm:classification_of_tunable_transitive_frames}, we obtain a positive answer to this question:

\begin{prop}
    For a finite family $(\alpha_i)_{i\in I}$ of ordinals, the complex algebras of the direct products $\prod_{i \in I}(\alpha_i,\leq)$ and $\prod_{i \in I}(\alpha_i,<)$ are both locally finite.
\end{prop}

\begin{proof}
    The case of $\prod_{i \in I}(\alpha_i,\leq)$ is directly covered by Dickson's Lemma, so it suffices to show well-foundedness and $D$-freeness of $\prod_{i \in I}(\alpha_i,<)$.
    Let $\prec$ be the relation on $\prod_{i \in I}(\alpha_i,<)$ and let $\preccurlyeq$ be the relation on $\prod_{i \in I}(\alpha_i,\leq)$.
    Since $\prec \subseteq \preccurlyeq$ and $\preccurlyeq$ is well-founded, $\prec$ is also well-founded.
    To show that $\prec$ is also $D$-free, let $\preceq$ denote the reflexive closure of $\prec$, and let $\{C_i\}_{i\geq 1}$ be any family of chains in $(\prod_{i \in I}\alpha_i,\preceq)$ such that $C_i \cong (i,\leq_i)$. 
    Let $b_i$ be the bottom element of $C_i$. 
    Then by Dickson's Lemma, $\{b_i\}_{i\geq 2}$ does not form an antichain with respect to $\preccurlyeq$, i.e., there is some $b_j,b_k$ with $b_j\preccurlyeq b_k$ and $j \neq k$.
    But $C_k$ has at least two elements, so there is some $b_k\neq c_k \in C_k$ with $b_k\prec c_k$. Hence $b_j\preccurlyeq b_k \prec c_k$, implying $b_j \prec c_k$.
    Thus $\bigcup_{i\geq 1}C_i$ is not isomorphic to $D$, implying $\prod_{i \in I}(\alpha_i,<)$ is $D$-free, as desired.
\end{proof}

\section{Acknowledgments}

This work was supported by NSF Grant DMS - 2231414.
I am sincerely grateful to Ilya Shapirovsky for introducing me to the study of tunable frames and locally finite algebras, and for the invaluable discussions and suggestions made throughout the development of this work.

\nocite{*}
\bibliographystyle{eptcs}
\bibliography{generic}
\end{document}